\documentclass[11pt]{article}
\usepackage{amssymb}
\usepackage{amsthm}
\usepackage{amsmath}
\usepackage{graphicx}

 \newtheorem{thm}{Theorem}[section]
 
 \newtheorem{lem}[thm]{Lemma}
 
 \theoremstyle{definition}
 
 \newtheorem{rem}[thm]{Remark}
 \numberwithin{equation}{section}

\theoremstyle{definition}

\theoremstyle{remark}

\begin{document}
\title{Remarks of Global Wellposedness of Liquid
Crystal Flows and Heat Flows of Harmonic Maps in Two Dimensions}
\author{Zhen
 Lei\footnote{School of Mathematical Sciences; LMNS and Shanghai
 Key Laboratory for Contemporary Applied Mathematics, Fudan University,
 Shanghai 200433, P. R.China. {\it Email:
 leizhn@gmail.com}}\and Dong Li\footnote{Department of Mathematics,
 University of British Columbia, 1984 Mathematical Road
  Vancouver, BC V6T 1Z2. {\it Email: mpdongli@gmail.com}}\and Xiaoyi
Zhang\footnote{Department of Mathematics,
 University of Iowa, 14 MacLean Hall, Iowa City, USA 52242. {\it Email:
 zh.xiaoyi@gmail.com} }}
\date{\today}
\maketitle

\begin{abstract}
We consider the Cauchy problem to the two-dimensional incompressible
liquid crystal equation and the heat flows of harmonic maps
equation. Under a natural geometric angle condition, we give a new
proof of the global well-posedness of smooth solutions for a class
of large initial data in energy space. This result was originally
obtained by Ding-Lin in \cite{DingLin} and Lin-Lin-Wang in
\cite{LinLinWang}. Our main technical tool is a rigidity theorem
which gives the coercivity of the harmonic energy under certain
angle condition. Our proof is based on a frequency localization
argument combined with the concentration-compactness approach which
can be of independent interest.
\end{abstract}

\maketitle





\section{Introduction}

We consider the following hydrodynamic system modeling the flow of
nematic liquid crystal materials in two dimensions (see, for
instance, \cite{Erickson62, Leslie, LiuLin95}):
\begin{equation}\label{LC}
\begin{cases}
u_t + (u\cdot\nabla) u + \nabla p = \Delta u - \nabla\cdot(\nabla
d\otimes \nabla d),\\[-4mm]\\
d_t + (u\cdot\nabla) d = \Delta d + |\nabla d|^2d,\\[-4mm]\\
\nabla\cdot u = 0,\quad |d| = 1,
\end{cases}
\end{equation}
where $u$ is the velocity field, $p$ is the scalar pressure and
$d=(d_1,d_2,d_3)$ is the unit-vector ($|d| = \sqrt{d_1^2+d_2^2+d_3^2} =1$) on the sphere $\mathbb{S}^2 \subset
\mathbb{R}^3$ representing the macroscopic molecular orientation
of the liquid crystal materials. Here the $i^{th}$ component of
$\nabla\cdot(\nabla d\otimes \nabla d)$ is given by
$\nabla_j(\nabla_i d\cdot \nabla_jd) = \sum_{j=1}^3 \nabla_j(\nabla_i d\cdot \nabla_jd) $. We apply Einstein's
summation convention over repeated indices throughout this paper.
For simplicity, we have set all the positive constants in the
system to be one. We are interested in the Cauchy problem of
\eqref{LC} with the initial data
\begin{equation}\label{data}
u(0, x) = u_0(x),\quad d(0, x) = d_0(x), \qquad \lim_{|x|\to \infty}
d_0(x) = \hat n \in \mathbb S^2,
\end{equation}
where $\hat n$ is a given unit vector.

The above system \eqref{LC} is a simplified version of the
Ericksen-Leslie model for the hydrodynamics of nematic liquid
crystals \cite{Erickson62, Leslie}. The mathematical analysis of
the liquid crystal flows was initiated by Lin and Liu in
\cite{LiuLin95, LiuLin2}. In \cite{Wang}, Wang proved the global existence of
strong solutions for rough initial data  with sufficiently
small ${\rm BMO}^{-1}$ norm (for $u_0$ and $\nabla d_0$). See
also \cite{LiWang} for a small data global existence result in 3D.
For any bounded smooth domain in $\mathbb R^2$, Lin, Lin and Wang \cite{LinLinWang} obtained
 the existence of global weak solutions which are smooth everywhere except on finitely many time slices
 (see also \cite{Hong}). The uniqueness of weak solutions
in two dimensions was studied by \cite{LinWang, Zhang}. Very
recently, a family of non-trivial two-dimensional exact large
solutions was constructed in \cite{DongLei}.

In this paper, we are concerned with the global existence of large classical solutions to \eqref{LC}.  As always with these types of problems
to extend the smooth local solution globally in time one needs to obtain certain a priori estimates.  By the regularity theory in
 \cite{LinLinWang}\footnote{Although the main results in \cite{LinLinWang} are stated for the bounded domains in $\mathbb R^2$,
 it is not difficult to check that the arguments there carry over to the $\mathbb R^2$ case.},  the smooth local solution
 $(u,d)$ to \eqref{LC}  can
 be continued past any time  $T>0$ provided that we have
 \begin{align}
 \int_0^T \Bigl( \| u(t, \cdot ) \|_{L^4}^4 + \| \nabla d(t, \cdot) \|_{L^4}^4 \Bigr) dt <\infty.  \label{criteria}
 \end{align}

The basic energy inequality associated with \eqref{LC} is the
following:
\begin{align}\label{energy}
\frac{1}{2}\big(\|u(t, \cdot)\|_{L^2}^2 + \|\nabla d(t,
\cdot)\|_{L^2}^2\big) & + \int_0^t\Bigl(\|\nabla u(s,\cdot)
\|_{L^2}^2 + \|\Delta d(s,\cdot) + |\nabla d|^2d(s,\cdot)
\|_{L^2}^2\Bigr)ds\\\nonumber \leq&\;
\frac{1}{2}\big(\|u_0\|_{L^2}^2 + \|\nabla d_0\|_{L^2}^2\big),
\qquad\forall\, t \ge 0.
\end{align}
By \eqref{criteria}--\eqref{energy} and the Gagliardo-Nirenberg inequality
\begin{align*}
\| u \|_{L^4(\mathbb R^2)} \le C \cdot \| u\|_{L^2(\mathbb R^2)}^{\frac 12} \cdot \| \nabla u \|_{L^2(\mathbb R^2)}^{\frac 12},
\end{align*}
 it is obvious that
 \begin{align*}
 \| u\|_{L_{t,x}^4([0,T) \times \mathbb R^2)}
  & \le C \cdot \| u \|_{L_t^{\infty} L_x^2([0,T) \times \mathbb R^2)}^{\frac 12}   \cdot \| \nabla u \|_{L_{tx}^2([0,T) \times \mathbb R^2)}^{\frac 12}
  \notag \\
& \le C \cdot  (\|u_0\|_{L^2}+ \| \nabla d_0 \|_{L^2}) <\infty, \qquad \text{for any } T>0.
\end{align*}
Hence the non-blowup criteria \eqref{criteria} can be sharpened to
\begin{align} \label{criteria1}
\int_0^T \| \nabla d(t, \cdot) \|_{L^4}^4 dt <\infty.
\end{align}

Note that in \eqref{LC} if we take $u \equiv 0$ we obtain the usual
heat flows of harmonic maps. The condition \eqref{criteria1} asserts
that the regularity of the whole liquid crystal system \eqref{LC} is
in some sense governed by the pure heat harmonic flow part. For
smooth solutions, the main obstruction in obtaining the a priori
estimate \eqref{criteria1} is that the harmonic energy $E_{harm}(t)
= \| \Delta d(t) + |\nabla d(t) |^2 d(t) \|_{2}^2$ in \eqref{energy}
is not coercive. In particular it yields no useful lower bounds for
general maps $d:\; \mathbb R^2 \to  \mathbb S^2$ (a trivial example
is the constant map).  A natural idea is to introduce some geometric
constraints and work with a set of "restricted maps"  such that the
resulting harmonic energy is coercive.  This idea was first used in
the seminal work of
 Ding-Lin in \cite{DingLin} when they were studying a generalization of
the Eells-Sampson's theorem.  Inspired by \cite{DingLin},  we follow
this line of thought and establish global wellposedness for a family
of initial data under a geometric angle condition. Roughly speaking,
this condition says that the image of the initial orientation vector
$d_0$ is strictly contained in a hemisphere. The set of such maps is
invariant under the dynamics of \eqref{LC} thanks to the maximum
principle.

\begin{thm}\label{thm1}
Denote $e_3=(0,0,1) \in \mathbb S^2$. Let $u_0, \nabla d_0 \in L^2(\mathbb{R}^2)$
with $|d_0| = 1$, $d_0
- e_3 \in L^2(\mathbb{R}^2)$ and satisfy the geometric angle condition:
\begin{align}
\inf_{x\in \mathbb{R}^2}d_{03}
> 0, \label{g_angle}
\end{align}
where $d_{03}$ is the third component of $d_0$.  Then there exists a unique global smooth solution $(d,u)$ to the
incompressible liquid crystal equations \eqref{LC} with the
initial data \eqref{data}. Moreover, one
has
\begin{eqnarray}\label{5}
\int_0^\infty \Bigl( \| \Delta d(t, \cdot) \|_{L^2}^2 + \|\nabla d(t, \cdot)\|_{L^4}^4+ \|\nabla d(t,\cdot) \|_{L^2}^2 \Bigr) dt < \infty.
\end{eqnarray}
\end{thm}

\begin{rem}
We emphasize that the result in Theorem \ref{thm1} (and also Theorem \ref{thm2} below)
 can  be obtained by using the method of \cite{LinLinWang}.\footnote{We are grateful to Prof. Fanghua Lin for
 pointing this out to us.}  Our proof here is
  based on a frequency localization argument combined with the concentration-compactness approach
which can be of independent interest.
\end{rem}
\begin{rem}
In Theorem \ref{thm1}, the choice of $e_3$ is for convenience only.
In general one can choose any reference vector $\widehat{n} \in
\mathbb S^2$ and require that the image of $d$ is contained in a
hemisphere around $\widehat{n}$.  We should also point it out that
the geometric angle condition \eqref{g_angle} may be necessary due
to some counterexamples in heat flows of harmonic maps. In
\cite{ChangDingYe-1} the authors proved the global well-posedness of
large solutions to the heat flows of harmonic maps for a class of
initial data with symmetry:
\begin{equation}\nonumber
d_0(x) = \begin{pmatrix}
           x_1r^{-1}\sin\psi_0(r) \\
           x_2r^{-1}\sin\psi_0(r) \\
           \cos\psi_0(r) \\
         \end{pmatrix}
\end{equation}
where
\begin{equation}\nonumber
\psi_0(0) = 0,\ \psi_0(R) < \pi\ {\rm for\ all}\  R > 0.
\end{equation}
Finite time singularities are also shown to exist in the case that
$\psi_0(R) > \pi\ {\rm for\ some}\  R > 0$ (see
\cite{ChangDingYe-2} for more details).
\end{rem}

\begin{rem}
As was shown in the previous remark, there exists a family of blowup
solutions for the heat flows of harmonic maps when the geometric
angle condition is violated. These blowup examples furnish a trivial
($u\equiv0$) set of counterexamples for the liquid crystal system
\eqref{LC}.  A very interesting open question is whether there are
finite time singularities for the incompressible liquid crystal
flows in two dimensions with finite energy and nontrivial velocity.
We note that such nontrivial counterexamples (if they exist) must be
non-radial since the incompressible constraint in the momentum
equation of \eqref{LC} destroys the spherical symmetry.
\end{rem}

By standard local theory (cf.  \cite{LinLinWang} and \cite{SU, Struwe}),
there exists a local smooth solution to
\eqref{LC} with the initial data \eqref{data} under the
assumptions stated in Theorem \ref{thm1}.
 Moreover, by the
weak-strong uniqueness (cf. \cite{LinWang} and
\cite{Zhang}), such a solution is also unique. For simplicity of presentation we will
refrain from proving such local wellposedness results and only work with smooth
local solutions.  The main work then is to prove the a priori estimate \eqref{criteria1}.
For this we will invoke the geometric angle assumption and prove that the corresponding harmonic
energy is coercive.  A crucial step is to establish the rigidity of approximate harmonic
maps (cf. \eqref{e1248}) under the geometric angle condition. We state it as the following

\begin{thm}\label{rigidity}
Let $\epsilon_0 > 0$, $C_0 > 0$. There exists a positive constant
$\delta_0 = \delta_0(\epsilon_0, C_0)$ such that the following
holds:

If $d: \mathbb{R}^2 \rightarrow \mathbb{S}^2$, $\nabla d \in
H^1(\mathbb{R}^2)$ with $\|\nabla d\|_{L^2} \leq C_0$ and $d_{3}
\geq \epsilon_0$, then
\begin{align}
\| \nabla d \|_{L^4}^4 \le (1-\delta_0) \| \Delta d \|_{L^2}^2. \notag
\end{align}

Consequently for such maps the associated harmonic energy is coercive, i.e.
\begin{align} \label{harm_coer}
\|\Delta d + |\nabla d|^2d\|_{L^2}^2 \geq \frac {\delta_0} 2
\Bigl( \|\Delta d\|_{L^2}^2 + \| \nabla d \|_{L^4}^4 \Bigr).
\end{align}
\end{thm}

\begin{rem}
Applying the method of \cite{DingLin} one can even obtain  a stronger rigidity
theorem.  We thank Prof. Fanghua Lin for pointing this out. Our approach
here is based on frequency localization combined with concentration-compactness.
\end{rem}

Note that our rigidity Theorem \ref{rigidity} is purely "static" and
it has nothing to do with the velocity field of the fluid dynamics.
As such it is stated for any map $d: \mathbb{R}^2 \rightarrow
\mathbb{S}^2$ under general assumptions. By using essentially the
same arguments,  we also have a similar result for the heat flows of
harmonic maps:
\begin{equation}\label{HHM}
d_t  = \Delta d + |\nabla d|^2d,
\end{equation}
where $d$ is still a unit-vector on the sphere $\mathbb{S}^2
\subset \mathbb{R}^3$ (see \cite{DingLin} for a generalization of
Eells-Sampson's theorem). We state the result as the following
\begin{thm}\label{thm2}
Let $\nabla d_0 \in L^2(\mathbb{R}^2)$ with $|d_0| = 1$, $d_0 - e_3
\in L^2(\mathbb{R}^2)$ and $\inf_{x\in \mathbb{R}^2}d_{03} > 0$.
Then there exists a unique global smooth solution to the heat flows
of harmonic maps \eqref{HHM} with the initial data $d(0, x) =
d_0(x)$. Moreover, \eqref{5} is satisfied.
\end{thm}

The proof of Theorem \ref{thm2}  will be omitted since it is essentially a repetition of the proof of
Theorem \ref{thm1}.

We close this introduction by
 setting up some notations and conventions which will be used in this paper.

\subsubsection*{Notations}

  For any two quantities $A$ and $B$,  we use $A \lesssim B$ (resp. $A
\gtrsim B$ ) to denote the inequality $A \leq CB$ (resp. $A \geq
CB$) for a generic positive constant $C$. The dependence of $C$ on
other parameters or constants are usually clear from the context
and we will often suppress  this dependence. The value of $C$ may
change from line to line.  For any function $f:\; \mathbb R^2\to
\mathbb R$, we use $\|f\|_{L^p}$ or sometimes $\|f\|_p$ to denote the  usual Lebesgue
$L^p$ norm of a function for $1 \le p \le \infty$.
We use $L_t^pL_x^r$
to denote the space-time norm
\begin{align*}
\| f\|_{L_t^pL_x^r} =\Bigl( \int_{\mathbb R} \bigl( \int_{\mathbb R^2} |f(t,x)|^r dx
\bigr)^{\frac pr} \Bigr)^{\frac 1p},
\end{align*}
with the usual modifications when $p$ or $r$ is infinity, or when the
domain $\mathbb R\times \mathbb R^2$ is replaced by a space-time slab.
When $p=r$ we abbreviate $L_t^p L_x^r$ by $L_{t,x}^p$ or $L_{tx}^p$.

Occasionally we
will need use the fractional Laplacian operator $|\nabla|^s$,
$s>0$, which is defined via Fourier transform as
\begin{align*}
\mathcal F (  |\nabla|^s f ) (\xi) = |\xi|^s (\mathcal F f)(\xi), \qquad \xi \in \mathbb R^2.
\end{align*}
The homogeneous Sobolev norm $\dot H^{s}$ for any $s>0$ is defined as
$\| f \|_{\dot H^{s} } = \| |\nabla|^s f \|_{2}$ or more explicitly:
\begin{align*}
\| f \|_{\dot H^s} = \Bigl( \int_{\mathbb R^2} |\xi|^{2s}  |(\mathcal F f)(\xi)|^2 d\xi \Bigr)^{\frac 12}.
\end{align*}

We will need to use the Littlewood-Paley (LP) frequency projection operators. For simplicity
we shall fix the notations on $\mathbb R^2$, but it is straightforward to define everything
in $\mathbb R^d$ for any $d\ge 1$.
To fix the notation let $\phi \in C_0^\infty(\mathbb{R}^2)$ and satisfy
\begin{equation}\nonumber
0 \leq \phi \leq 1,\quad \phi(x) = 1\ {\rm for}\ |x| \leq 1,\quad
\phi(x) = 0\ {\rm for}\ |x| \geq 2.
\end{equation}
For two real positive numbers $\alpha < \beta$, define the
frequency localized (LP) projection operator $P_{\alpha <\cdot<\beta}$ by
\begin{equation}\nonumber
P_{\alpha <\cdot<\beta}f =
\mathcal{F}^{-1}\big([\phi(\beta^{-1}\xi) -
\phi(\alpha^{-1}\xi)]\mathcal{F}(f)\big).
\end{equation}
Here $\mathcal{F}$ and $\mathcal{F}^{-1}$ denote the Fourier
transform and its inverse transform, respectively. Similarly, the
operators $P_{< \alpha}$ and $P_{> \beta}$ are defined by
\begin{equation}\nonumber
P_{<\beta}f =
\mathcal{F}^{-1}\big(\phi(\beta^{-1}\xi)\mathcal{F}(f)\big),
\end{equation}
and
\begin{equation}\nonumber
P_{> \alpha}f = \mathcal{F}^{-1}\big([1 -
\phi(\alpha^{-1}\xi)]\mathcal{F}(f)\big).
\end{equation}

We recall the following Bernstein estimates:  for any $1\le p\le
q\le \infty$ and dyadic $N>0$,
\begin{align*}
\| P_{<N} f\|_{L_x^q (\mathbb R^2)} \lesssim N^{\frac 2p -\frac 2q}
\| f\|_{L_x^p(\mathbb R^2)}.
\end{align*}


\section{Rigidity}

In this section we prove our rigidity Theorem \ref{rigidity}.  Our proof is
purely analytic, and uses in a quantitative way the geometric angle condition.
The key ingredient of the proof is a  frequency localization
argument and a concentration-compactness procedure.



We begin with the following simple lemma which locks the nontrivial
$L^2$ weak limit of an $L^2$ bounded  sequence of functions whose frequency is essentially
localized to a (large) ring. The proof is stated for $\mathbb R^2$ but it naturally generalizes to $\mathbb R^d$ for any dimension
$d\ge 1$.

\begin{lem}[Non-evacuation of $L^2$-norm]
\label{lem111}
Let $C_0$ and $\alpha$ be two positive constants, and $N$ be a
dyadic integer. Suppose that $f_n: \; \mathbb R^2 \to \mathbb R$
is a sequence of functions such that
\begin{align*}
\| f_n \|_{L^2} \le C_0, \quad\forall\; n,
\end{align*}
and
\begin{align}
\| P_{\frac 1 N<\cdot <N} f_n \|_{L^\infty} \ge \alpha>0,
\quad\forall\, n. \label{e1200}
\end{align}
Then one can find a  subsequence (which we still label as $f_n$) and
centers $x_n \in \mathbb R^2$, such that
\begin{align*}
f_n(x_n - \cdot) \rightharpoonup f \quad\text{weakly in $L^2$},
\end{align*}
and
\begin{align*}
\| f \|_{L^2} \gtrsim 1.
\end{align*}

\end{lem}

\begin{proof}
By \eqref{e1200},  we can find $x_n \in \mathbb R^2$ such
that
\begin{align*}
| (P_{\frac 1 N<\cdot <N} f_n)(x_n) | \ge \frac {\alpha}2,
\quad\forall\, n.
\end{align*}
Or more precisely, for $\psi = \mathcal{F}^{-1}\big(\phi(N^{-1}\xi)
- f(N\xi)\big) \in \mathcal S(\mathbb R^2)$, one has
\begin{align}
|\int_{\mathbb R^2} \psi(y) f_n(x_n -y) dy | \ge \frac {\alpha}
2,\quad\forall\, n. \label{e1203}
\end{align}
Since by assumption the sequence $f_n(x_n-\cdot)$ is uniformly
bounded in  $L^2$, there exists a subsequence (we still denote it by
$f_n(x_n-\cdot)$) which converges weakly in $L^2$ to some function
$f$. Clearly, by taking the limit in \eqref{e1203} and using the
Cauchy-Schwartz inequality, one concludes that $\|f\|_{L^2} \gtrsim
1$.
\end{proof}

Our next lemma allows us to remove the translational degrees of
freedom in studying the rigidity property of approximate harmonic
maps from $\mathbb R^2$ to $ \mathbb S^2$ under the geometric angle
condition. By removing the translational degrees of freedom (and
quotienting out other possible non-compact group actions), we can
restore the compactness in the same spirit as the usual
concentration-compactness procedure.

\begin{lem}[Removing translational invariance] \label{lem120}
Let $C_0$ and $\alpha$ be two positive constants. Suppose a
sequence of maps $d_n: \; \mathbb R^2 \to \mathbb{S}^2$, $n\ge 1$ satisfy
the following conditions:
\begin{itemize}
\item $\| \nabla d_n \|_{L^2} + \| \nabla d_n \|_{\dot{H}^1} \le
C_0<\infty, \quad \forall\, n,$ \item $ \| \nabla d_n \|_{L^4}^2
\ge \alpha >0, \quad\forall\, n$.
\end{itemize}
Then, upon a subsequence in $n$ if necessary, we can find a
sequence of points $x_n \in \mathbb R^2$ such that
\begin{align*}
| \nabla d_n(x_n-\cdot ) |^2 \rightharpoonup f, \quad\text{weakly in
$L^2$}
\end{align*}
and $\| f \|_{L^2} \gtrsim 1$.
\end{lem}

\begin{proof}
Let  $N$ be a dyadic integer which will be taken sufficiently large in the course of the proof. Our main idea is to localize
the $\nabla d_n$-sequence within the frequency window $[1/N, N]$.  After that we apply Lemma \ref{lem111} to find
the nontrivial weak limit.

We first deal with the high frequency piece.  By frequency localization, we have
\begin{align*}
 P_{>N}( (\nabla d_n)^2 )= P_{>N}( P_{>N/8} \nabla d_n \cdot P_{>N/8} \nabla d_n)+ 2P_{>N} (P_{\le N/8} \nabla d_n \cdot P_{>N/8} \nabla d_n).
\end{align*}
By Sobolev embedding, we then have
\begin{align}
\| P_{>N} ( (\nabla d_n)^2 ) \|_{L^2} & \lesssim  \| P_{>N/8} \nabla d_n \|_{L^4} \| \nabla d_n \|_{L^4} \notag \\
& \lesssim  \| |\nabla|^{\frac 12} P_{>N/8} \nabla d_n \|_{L^2} \cdot C_0 \notag \\
& \lesssim  N^{-1/2}  \| \Delta d_n \|_{L^2} \cdot C_0 \notag \\
& \lesssim N^{-1/2} C_0^2 \le \alpha/10 , \label{tmp100}
\end{align}
if we take $N$ large enough.

Similarly for the low frequency part, we use Bernstein's inequality to get
\begin{align}\label{2}
\| P_{<\frac 1N} ( (\nabla d_n)^2 ) \|_{L^2} \le N^{-1}\| (
(\nabla d_n)^2 ) \|_{L^1} \leq \alpha/10,
\end{align}
where again we need to take $N$ large enough.

Now using \eqref{tmp100} and \eqref{2} and the assumption $\|(\nabla
d_n)^2\|_{L^2} \geq \alpha$, we obtain
\begin{align*}
\|P_{\frac{1} {N} < \cdot< N} ( (\nabla d_n)^2 ) \|_{L^2} \geq
\frac {\alpha} 2.
\end{align*}
On the other hand, by H\"older, one has
\begin{align*}
&\|P_{\frac {1} {N} < \cdot< N} ( (\nabla d_n)^2 ) \|_{L^2} \\
& \lesssim \| (\nabla d_n)^2 \|_{L^1}^{\frac 12}  \| \|P_{\frac {1}{N} < \cdot< N} ( (\nabla d_n)^2 ) \|_{L^\infty}^{\frac 12}\\
& \lesssim \|P_{\frac {1}{ N} < \cdot <N} ( (\nabla d_n)^2 )
\|_{L^\infty}^{\frac 12}.
\end{align*}
Obviously
\begin{align*}
\|P_{\frac {1}{ N} < \cdot <N} ( (\nabla d_n)^2 ) \|_{L^\infty}
\gtrsim 1.
\end{align*}
We can then apply Lemma \ref{lem111} to the sequence $f_n= |\nabla
d_n|^2$ to get the result.
\end{proof}

Now we are ready to prove the rigidity Theorem \ref{rigidity}
under the geometric angle condition $d_{3} \geq \epsilon_0$.
\begin{proof}[Proof of Theorem \ref{rigidity}]
First of all, using  a scaling argument $d(x) \to d(\lambda x)$
with $\lambda = \|\Delta d\|_{L^2}^{-1}$, we may assume
\begin{equation}\nonumber
\|\Delta d\|_{L^2} = 1.
\end{equation}

It then suffices to show there exists $\delta_0 =
\delta_0(\epsilon_0, C_0)>0$ such that
\begin{equation}\label{3}
\| \nabla d \|_{L^4}^4 \le 1-\delta_0.
\end{equation}
Assume \eqref{3} does not hold. Then we can find a sequence $d_n: \mathbb{R}^2
\rightarrow \mathbb{S}^2$ such that
$$\|\Delta d_n\|_{L^2} = 1,\quad \|\nabla d_n\|_{L^2}
\leq C_0,\quad d_{n3} \geq \epsilon_0, $$ but $$\| \nabla d_n
\|_{L^4} \nearrow 1,\ {\rm as}\ n\to \infty.$$

Denote
\begin{align}
g_n = \Delta d_n + |\nabla d_n|^2 d_n. \label{e1248}
\end{align}
It is easy to compute that
\begin{align}
\|g_n\|_{L^2}^2 &= \|\Delta d_n\|_{L^2}^2 + \|\nabla
  d_n\|_{L^4}^4 + 2\int|\nabla d_n|^2\Delta d_n\cdot
  d_ndx   \notag \\
&= 1 - \|\nabla d_n\|_{L^4}^4. \label{e1248aaa}
\end{align}
Hence
\begin{eqnarray}\nonumber
\|g_n\|_{L^2} \searrow 0\quad {\rm as }\ n \rightarrow \infty.
\end{eqnarray}
Applying Lemma \ref{lem120} and performing a simple transform to
eliminate the translation invariance if necessary, we conclude that
\begin{eqnarray}\nonumber
|\nabla d_n |^2 \rightharpoonup  g(x),\quad {\rm weakly\ in\ L^2\
as}\ n \rightarrow \infty
\end{eqnarray}
for some $g(x)\ge 0$ and $\|g\|_{L^2} \gtrsim 1$.  Consequently,
by taking the limit $n \rightarrow \infty$ for the equation of
$d_{n3}$ in \eqref{e1248}, we have
\begin{align*}
\Delta d^*_{3} + g(x) \epsilon_0 \le 0,
\end{align*}
where $d_{3}^*$ is the weak limit of the sequence (upon
subsequence if necessary) $d_{n3}$ which satisfies
\begin{equation}\nonumber
\|\Delta d_{3}^*\|_{L^2} \leq 1,\quad \|\nabla d_{3}^*\|_{L^2}
\leq C_0.
\end{equation}
Now $d_{3}^*$ is superharmonic on $\mathbb R^2$ and is bounded.
Obviously $d_{3}^*$ must be a constant map. But since  $\Delta
d^*_{3} + g(x) \epsilon_0 \le 0$, one concludes that $g(x)$ must
be identically $0$. This contradicts to the fact $\|g\|_{L^2}
\gtrsim 1$.

We have obtained the desired contradiction and therefore \eqref{3} holds. Finally \eqref{harm_coer} follows
from \eqref{3}, and the same computation (with $d_n$ replaced by $d$) as in \eqref{e1248} and \eqref{e1248aaa}.

This concludes the proof of Theorem
\ref{rigidity}.
\end{proof}

\section{Proof of Theorem \ref{thm1}}
As was already mentioned in the introduction,  we only need to
establish the a priori estimate \eqref{criteria1}. By using the
maximum principle (applied to the third component of $d_3$), we
have
\begin{equation}\nonumber
\inf_{x \in \mathbb{R}^2}d_3 (t,x)\geq \inf_{x \in \mathbb{R}^2}d_{03}, \qquad \forall\, t
> 0.
\end{equation}
One can apply the rigidity Theorem \ref{rigidity} to conclude that
\begin{align}
\|\Delta d(t) + |\nabla d(t) |^2d(t)\|_{L^2}^2 \geq \frac {\delta_0} 2 \Bigl( \|\Delta
d(t) \|_{L^2}^2 + \| \nabla d(t) \|_4^4 \Bigr), \qquad \forall\, t>0,
\end{align}
where $\delta_0>0$ is the same constant as in Theorem \ref{rigidity}.

Consequently, the basic energy law \eqref{energy} yields
\begin{eqnarray}\label{4}
&&\frac 12 \| \nabla d (t) \|_{L^2}^2+ \frac {\delta_0}
  2\int_0^T\big(\|\Delta d (t) \|_{L^2}^2 + \|\nabla  d(t)
  \|_{L^4}^4 \big)dt\\\nonumber
&&\quad\quad \leq \frac{1}{2}\big(\|u_0\|_{L^2}^2
  + \|\nabla d_0\|_{L^2}^2\big), \qquad \forall\, t\ge 0.
\end{eqnarray}
This clearly settles \eqref{criteria1} and the first two terms in
the estimate \eqref{5}. It remains for us to prove the estimate
\begin{align}
\int_0^{\infty} \| \nabla d (t) \|_{L^2}^2 dt <\infty. \label{etmp146}
\end{align}

By performing the basic energy estimate for the equation of $d -
e_3$ and using the interpolation inequalities, we have
\begin{eqnarray}\nonumber
&&\frac{1}{2}\frac{d}{dt}\|d - e_3\|_{L^2}^2 + \|\nabla
  d\|_{L^2}^2\\\nonumber
&&\leq \int|\nabla d|^2d(d - e_3)dx \leq \|\nabla d\|_{L^4}^2\|d -
  e_3\|_{L^2}\\\nonumber
&&\lesssim \|\nabla d\|_{L^2}\|\Delta d\|_{L^2}\|d -
  e_3\|_{L^2}\\\nonumber
&&\leq \frac{1}{2}\|\nabla
  d\|_{L^2}^2 + C\|\Delta d\|_{L^2}^2\|d -
  e_3\|_{L^2}^2,
\end{eqnarray}
where $C>0$ is an absolute constant.  A Gronwall in time argument then gives
\begin{eqnarray}\nonumber
\|d(t,\cdot) - e_3\|_{L^2}^2 + \int_0^t\|\nabla d\|_{L^2}^2ds \lesssim
\|d_0 - e_3\|_{L^2}^2\exp\Big(\int_0^t\|\Delta d\|_{L^2}^2ds\Big), \qquad \forall\, t\ge 0.
\end{eqnarray}
Noting \eqref{4}, one concludes that
\begin{eqnarray}\nonumber
\int_0^t\|\nabla d\|_{L^2}^2ds < \infty, \qquad \forall\, t\ge 0.
\end{eqnarray}
This obviously implies \eqref{etmp146}.

This concludes the proof of Theorem \ref{thm1}.
Finally we point out that a  further higher order energy estimate also implies that any $H^s$
norm of $(u, \nabla d)$ is uniformly bounded in time for any $s >
0$.


\section*{Acknowledgement}
The authors would like to thank Professor Hongjie Dong and Professor
Fang-hua Lin for many helpful discussions. This work was done when
Dong Li and Xiaoyi Zhang were visiting the School of Mathematical
Sciences of Fudan University during 2011. They would like to thank
the hospitality of the school. Zhen Lei was supported by NSFC
(grants No.11171072 and 11222107), the Foundation for Innovative
Research Groups of NSFC (grant No.11121101), FANEDD, Innovation
Program of Shanghai Municipal Education Commission (grant
No.12ZZ012), NTTBBRT of NSF (No. J1103105), and SGST 09DZ2272900.
Dong Li was supported in part by NSF grant 0908032. D. Li was
supported in part by NSF under agreement No. DMS-1128155. Any
opinions, findings and conclusions or recommendations expressed in
this material are those of the authors and do not necessarily
reflect the views of the National Science Foundation. Xiaoyi Zhang
was supported in part by the Alfred Sloan Fellowship.


\frenchspacing
\bibliographystyle{plain}

\end{document}